\documentclass[10pt]{article}

\usepackage{fullpage}
\usepackage{amsmath}
\usepackage{eucal}

\usepackage{graphicx}
\usepackage{color}
\usepackage{latexsym}
\usepackage{amssymb}
\usepackage{booktabs}
\usepackage{url}
\usepackage{xspace}
\usepackage{pslatex}
\usepackage{textcomp}

\usepackage{amsthm}

\usepackage[ ruled, lined, noresetcount]{algorithm2e}

\newtheorem{lemma}{Lemma}

\newtheorem{observ}{Observation}

\newcommand{\old}[1]{{}}

\begin{document}

\newcommand{\NP}{\ensuremath{\mathcal{N}\mathcal{P}}\xspace}
\newcommand{\proofsketch}{\noindent\textbf{Proofsketch.}\enspace}
\newcommand{\bt}{\begin{tabbing}}
\newcommand{\et}{\end  {tabbing}}
\newcommand{\be}{\begin{enumerate}}
\newcommand{\ee}{\end  {enumerate}}
\newcommand{\bl}{\hspace*{2mm}}
\newcommand{\Bl}{\hspace*{5mm}}
\newcommand{\BL}{\hspace*{10mm}}
\newcommand{\sq}{\sqrt{2}}
\newcommand{\abs}[1]{{#1}}
\newcommand{\full}[1]{{}}

\def\registered{{\ooalign{\hfil\raise .00ex\hbox{\scriptsize R}\hfil\crcr\mathhexbox20D}}}

\title{Disruption Management with Rescheduling of Trips and Vehicle Circulations }

\old{
%%% first author
\author{Martin Lorek
    \affiliation{
%	Algorithms Group\\
	Department of Computer Science\\
	Braunschweig University of Technology\\
	Braunschweig, D-38106\\
  Germany\\
    Email: m.lorek@tu-bs.de
    }	
}

%%% second author
%%% remove the following entry for single author papers
%%% add more entries for additional authors
\author{S\'andor P.\ Fekete 
      % {\tensfb Alexander Kr{\"o}ller}     
    \affiliation{
%Algorithms Group\\
      Department of Computer Science\\
	Braunschweig University of Technology\\
	Braunschweig, D-38106\\
  Germany\\
    Email: s.fekete@tu-bs.de
    }
}

\author{Alexander Kr{\"o}ller     
    \affiliation{
%Algorithms Group\\
    Department of Computer Science\\
	Braunschweig University of Technology\\
	Braunschweig, D-38106\\
  Germany\\
    Email: a.kroeller@tu-bs.de
    }
}

\author{Marc E.\ Pfetsch       
    \affiliation{Department of Mathematics\\
	Braunschweig University of Technology\\
	Braunschweig, D-38106\\
  Germany\\
    Email: m.pfetsch@tu-bs.de
    }
}
}

\author{
S\'andor P.\ Fekete\thanks{ Department of Computer Science, TU
Braunschweig, D-38106 Braunschweig, Germany, Email:
\{s.fekete,m.lorek,a.kroeller,m.pfetsch\}@tu-bs.de. http://www.cs.ibr.tu-bs.de/alg}
 \and Alexander Kr{\"o}ller  \footnotemark[1] \and Martin Lorek \footnotemark[1] \thanks{Supported by a fellowship from Siemens AG.}
\and Marc Pfetsch 
\thanks{ Department of Mathematics, TU Braunschweig, D-38106 Braunschweig, Germany, Email: m.pfetsch@tu-bs.de.}
       }

\date{}
\maketitle
\markboth{}{}

%%%%%%%%%%%%%%%%%%%%%%%%%%%%%%%%%%%%%%%%%%%%%%%%%%%%%%%%%%%%%%%%%%%%%%%
%hey ho, let's go
\begin{abstract}
  This paper introduces a combined approach for the recovery of a
    timetable by rescheduling trips and vehicle circulations for a
    rail-based transportation system subject to disruptions. We propose a
    novel event-based integer programming (IP) model. Features include
    shifting and canceling of trips as well as modifying the vehicle
    schedules by changing or truncating the circulations.  The objective
    maximizes the number of recovered trips, possibly with delay, while
    guaranteeing a conflict-free new timetable for the estimated time
    window of the disruption.
    % the objective includes vehicle resources, safety margins, and a measure to return to the original timetable after a certain time horizon. 
    We demonstrate the usefulness of our approach through experiments for
    real-life test instances of relevant size, arising from the subway
    system of Vienna. We focus on scenarios in which one direction of one
    track is blocked, and trains have to be scheduled through this
    bottleneck. Solving these instances is made possible by contracting
    parts of the underlying event-activity graph; this allows a significant
    size reduction of the IP.
    Usually, the solutions found within one minute are of good quality and can be used as good estimates of recovery plans in an online context.
\end{abstract}

\section{Introduction}
\subsection{Disruption Management in Passenger Traffic}

Mobility of people is of growing importance in modern societies. For
instance, especially in densely populated European cities, public transportation
systems form a relevant economic factor. However, possible disruption of
operations in such systems will always remain unavoidable. Increased
pressure towards economical operation has amplified the impact and the
frequency of such disruptions. Being often optimized to the limit,
transportation systems have become more interconnected, and therefore more
vulnerable to disruptions.

%\vfill
%\begin{figure}[h!]
%  \centering
%  \includegraphics[width=.6\linewidth]{anzeigetafel}
%\end{figure}

This makes it indispensable to develop more advanced methods for disruption
management of passenger traffic. Railway disruption management is defined
in Jespersen-Groth et al. \cite{groth09} as the process of finding a new
timetable by rerouting, delaying, or canceling trains and rescheduling the
resources like rolling stock and the crew such that the new timetable is
feasible with respect to the new schedules. This has to be performed while
the effects of disruptions are still unfolding, that is, disruption
management is an online problem.

The recovery of the timetable, the vehicle schedule, and the crew schedule
is usually performed sequentially.  While previous work has focused on
either delay management (i.e., making sure that the delay of passengers and
the inconvenience caused by missed trains is kept low) or on deriving new
circulations on a given dispatching timetable, a particular aspect that has
yet to be considered from a scientific side is to combine rescheduling of
trips and vehicle circulations to obtain a feasible dispatching timetable.
The goal is to perform as many trips of the original timetable as possible.
Clearly, this is an important next step when further advancing optimization
towards real-time methods for stabilizing rail-based transportation
systems.

%This paper introduces a first approach towards the combined timetable recovery by retiming of trips and vehicle rescheduling in disruption management.
%We propose an appropriate,
%novel IP-based optimization model. By contracting parts of underlying graph we are able to reduce the size of the IP and solve a variety of real-life
%instances arising from the subway system of Vienna. Beyond the 
%particular mathematical methods, we believe that the problem itself
%and the benchmarks should motivate further
%work in transportation modeling and optimization.

\subsection{Related Work}

%\todo{more literauter: Liebchen, Huisman, Airline}
The challenges and problems connected with passenger railway transportation
have been intensively studied for the last decades, and operations research
methods have been successfully applied. For example, Liebchen investigated
the optimization of periodic timetables~\cite{Lie06}. However, in daily
operation, unforeseen events occur and may lead to disruptions of the
timetable. Different aspects of this problem have been considered, and
various approaches and models exist. A general framework for the problems
and solution methods in disruption management can be found
in~\cite{groth09}.  The problem we investigate in this paper is closely
related to the {\em delay management problem}. Mixed-integer approaches for
it are introduced by Sch\"obel~\cite{schoebel01,schoebel07} and extended by
capacity constraints for the tracks~\cite{schaschoe08}.
Schachtebeck~\cite{schachtebeck09} adds circulation constraints to the
model and performs a matching at the depots to prevent a large delay from
spreading into the next circulation activity. The computational complexity
of the original delay management problem is investigated
in~\cite{gatto_et_al05}: it turns out to be NP-hard, even in special cases,
see also~\cite{CapraraFischettiToth2002}. Simulation studies can be found
in~\cite{hofman_et_al06,suhl04,gely09}. T\"ornquist~\cite{trnquist06} gives
an overview over the existing models and decision support systems. He also
investigates the disturbance management on $n$-tracks~\cite{trnquist07}.
Gatto and Widmayer consider related problems as an online job shop
model~\cite{gatwid07}. A decision support system for rolling stock
rescheduling during disruption, where the rolling stock is balanced
according to a given dispatching timetable, is provided by
Nielsen~\cite{nielsen08}.  

In this paper, we focus on scenarios in which one direction of a track
network has to be shutdown, such that trains have to be scheduled through a
bottleneck. This leads to the track allocation problem studied by
Bornd\"orfer and Schlechte~\cite{BorndoerferSchlechte07}, which is shown to
be NP-hard in~\cite{CapraraFischettiToth2002}. A column generation approach
to timetabling trains though a corridor is presented by Cacchiani et
al.~\cite{CacchianiCapraraToth2008}.  Huisman et al.\ developed column
generation approaches to reschedule rail crew~\cite{Huisman07} and solved
the combined crew rescheduling problem with retiming of trips by allowing
small changes to the train timetables~\cite{HuismanVPK09}. An overview of
models and methods for similar problems that arise in the airline context
can be found in Clausen~\cite{Clausen10}.

None of the above papers considers rescheduling of trips and circulations
as an integral part of deriving a feasible dispatching timetable. Clearly,
we can imagine real-life scenarios that require an integrated approach for
disruption management, e.g., in an online situation with trains already
circulating in the rail network and an infeasible timetable due to heavy
disruptions.

%The high frequency of trains circulating in the system, the limited capacity of tracks in- and outside the stations and the bottleneck
%section force us to add stronger resource restrictions and integrate further
%possibilities into the model in order to obtain a feasible 
%dispatching timetable.
%None of these papers consider rescheduling of rolling stock as an integral part of deriving an feasible dispatching timetable.
%rescheduling has only been considered separately as a step towards achieving 
%a new, given schedule. 
%Unfortunately, none of these approaches use rescheduling as an
%integral part of the planning process, but rather use it for reaching a pre-computed schedule.
%Clearly, we can imagine real-life scenarios for that disruption management requires making use of rescheduling as integral
%part of the planning process. Especially in subway systems the utilization of the system is very high with lots of trains circulating in the system. Any sudden disturbance can lead to significant disruptions. The dispatchers have to cope with theses disruptions and have to provide a good service to the passengers.

\subsection{Our Contribution}

In this paper, we introduce a combined approach for re\-scheduling trips
and vehicle circulations in order to deal with unplanned disruptions. We
derive a new re-optimized dispatching timetable with respect to a feasible
new vehicle schedule that performs as much of the original trips as
possible. Especially in subway systems, the utilization of the system is
high. The high frequency of trains circulating in the system, the limited
possibilities for overtaking, swapping tracks, and temporarily parking, as
well as the bottleneck sections caused by disruptions force us to add
stronger resource restrictions and integrate further possibilities into the
model in order to obtain a feasible dispatching timetable.  Our main
mathematical contributions are an appropriate model based on integer
programming that combines rescheduling of trips and circulations, as well
as a mathematical contraction technique for simplifying the model. On the
practical side, we are able to demonstrate the usefulness of this
mathematical optimization approach by providing optimal and near-optimal
solutions for a system of real-life instances that are based on the subway
system of Vienna.

The rest of this paper is organized as follows. In the following section,
we describe detailed aspects of disruptions in the operation of rail-based
transportation. This leads to an optimization model based on integer linear
programming.  Mathematical aspects of solving real-life instances are
discussed, and computational results are presented before we conclude with a
discussion of future work.

\section{Disruption in Public Transport}

Disturbances in the daily operation of public passenger transport cannot be
avoided. This can include delays caused by longer waiting times, because an
unexpected number of passengers enter or leave the train, a driver arriving
late at a relief point, or a technical problem with the mechanism of an
automatic door.  These disruptions lead to small delays, which usually can
be absorbed by the buffer times integrated into the timetable. If bigger
disruptions occur, e.g., a vehicle breaks down or a part of the track
requires unplanned work, this causes a temporary shutdown on a track
section.  Such disruptions lead to severe violations of the timetable,
which have to be resolved by the dispatchers.

%, so that
%dispatchers need additional strategies to cope with this situation.
%In these cases the initial schedule becomes infeasible
%and cannot be repaired efficiently by only delaying trips. 

As an illustrating example, consider a blocked track section on a $2$-track
network. A possible option includes rerouting trips, which originally use
the blocked track, via the opposite track. This leads to a large number of
conflicts on that track section, because trains driving into opposite
directions have to share the same track resources and compete for free time
slots to pass trough this bottleneck. Consider two trains that are
scheduled to enter the shared track section in opposite directions at the
same time. Assume that the minimum transit time is 5 minutes and the
minimum safety time between two trains using the same track in opposite
direction is one minute. If the cycle time is 5 minutes and we allow a
maximum delay of 5 minutes, it is not possible to schedule both trains
through the bottleneck.

Depending on the length of the bottleneck and the frequency of the
railway system, updating the timetable by delaying trips will not suffice. If 
trains queue up in front of the bottleneck, passenger travel time will increase
dramatically. 
%
%\todo{
%umstrukturieren\\
%freiheitsgrade erl\"autern\\
%schwierigkeiten, insbesondere bei einspuriger Sperrung erl\"autern\\
%}
More precisely, the dispatcher has different possibilities for dealing with
this situation. The timetabling perspective includes that
\begin{enumerate} 
\item trips can be delayed within a certain time window and
\item trips can be canceled.
\end{enumerate}
In addition rescheduling of the vehicle circulations allows that
\begin{enumerate}
\item trains can be instructed to truncate their circulation and turn
  early, in order to serve trips initially planned for a different train,
\item trains can be used to shuttle inside the single-track sections,
\item trains can return early to a depot, and
\item replacement vehicles can be used. 
\end{enumerate}

Usually, timetabling and vehicle scheduling are planned sequentially.
Recently, research has focused more and more on integrating different steps of
the planning process, as proposed in~\cite{Lie08}. If heavy disruptions
occur, these two aspects are strongly interweaved and need to be handled at
the same time.
%Especially in subway systems with a high utilization of the network system and very limited possibilities for overtaking, swapping tracks and temporarily parking a feasible new dispatching timetable is needed very soon to prevent the collapse of the system. %Furthermore, the originally planned timetable should be recovered as soon as possible.

In the following, we present a model based on integer linear programming
that finds a feasible disposition timetable and new vehicle schedules. It
respects the capacity constraints of the tracks and allows dropping trips
and early turns. The goal is to return to the initial timetable within a
certain time horizon.

%%%%%%%%%%%%%%%%%%%%%%%%%%%%%%%%%%%%%%%%%%%%%%%%%%%%%%%%%%%%%%%%%%%%%%

\def\Edep{\mathcal{E}_{\text{dep}}}
\def\Earr{\mathcal{E}_{\text{arr}}}
\def\Athead{\mathcal{A}_{\text{t-head}}}
\def\Ashead{\mathcal{A}_{\text{s-head}}}
\def\Adrive{\mathcal{A}_{\text{drive}}}
\def\Await{\mathcal{A}_{\text{wait}}}
\def\Atrain{\mathcal{A}_{\text{train}}}
\def\Areturn{\mathcal{A}_{\text{return}}}
\def\Astation{\mathcal{A}_{\text{station}}}
\def\Edepot{\mathcal{E}_{\text{depot}}}
\def\Erdepot{\mathcal{E}_{\text{repl}}}
\def\Aturn{\mathcal{A}_{\text{turn}}}
\def\Atrack{\mathcal{A}_{\text{track}}}
\def\Ahead{\mathcal{A}_{\text{head}}}

%\begin{nomenclature}
%%\entry{A}{You may include nomenclature here.}
%%\entry{$\alpha$}{There are two arguments for each entry of the nomemclature environment, the symbol and the definition.}
%%\entry{}{}
%\entry{$\mathcal{S}$}{stations}
%\entry{$\mathcal{J}$}{tracks}
%\entry{$\mathcal{D}$}{depots}
%\entry{$\mathcal{T}$}{trips}
%\entry{$G=(\mathcal{S},\mathcal{J})$}{rail network}
%\entry{$\mathcal{V}$}{vehicles}
%\entry{$\mathcal{N}$}{event-activity network}
%\entry{$\Edep$}{departure events}
%\entry{$\Earr$}{arrival events}
%\entry{$\Edepot$}{arrival events at a depot}
%\entry{$\Erdepot$}{events for insertion of a replacement vehicle}
%\entry{$\mathcal{E}=\Edep\cup\Earr\cup\Edep\cup\Erdepot$}{all events}
%\entry{$\Adrive$}{driving activities}
%\entry{$\Await$}{waiting activities}
%\entry{$\Aturn$}{turning activities}
%\entry{$\Areturn$}{return to a depot activities}
%\entry{$\Atrain=\Adrive\cup\Await\cup\Aturn$}{train activities}
%\entry{$\Atrack$}{headway activities on the tracks between two stations}
%\entry{$\Astation$}{headway activities on a train platform}
%\entry{$\Ahead=\Atrack\cup\Astation$}{all headway activities}
%\entry{$\mathcal{A}=\Atrain\cup\Ahead\cup\Areturn$}{all activities}
%\entry{$\pi$}{timetable}
%\entry{$\pi_v$}{the scheduled time of event $v$}
%\entry{$x_v$}{delay of event $v$}
%\entry{$c_a$}{cost coefficient}
%\entry{$L_{a}$}{duration of activity $a$}
%\entry{$S_a$}{safety margin}
%\entry{$M$}{a large constant}
%
%\end{nomenclature}

%%%%%%%%%%%%%%%%%%%%%%%%%%%%%%%%%%%%%%%%%%%%%%%%%%%%%%%%%%%%%%%%%%%%%%%
\section{Our Model}
\subsection{A Graph Representation}
In the following we consider a rail network $G=(\mathcal{S},\mathcal{J})$ with $\mathcal{S}$ and $\mathcal{J}$ representing the set of stations and tracks and a given dispatching timetable $\pi$ that includes the schedule for each vehicle in $\mathcal{V}$.
Our model for disruption management 
with rescheduling of trips and vehicle circulations
is based on the widely used concept of an event-activity network. This structure was suggested
by Serafini and Ukovich \cite{Serafini89PESP} and also used by Nachtigall~\cite{nachtigall98} for periodic timetabling problems, see also \cite{opitz09}.
In the context of delay management problems event-activity networks are used in~\cite{schoebel07}.
We use a model similar to \cite{schachtebeck09} and extend the
event-activity network to account for early turnarounds and capacity 
constraints at the station platforms.

An event-activity
$\mathcal{N}=(\mathcal{E},\mathcal{A})$ is a directed graph, where the nodes in
$\mathcal{E}$ denote the events and the edges in $\mathcal{A}$ are called
activities. 
In our setting the timetable $\pi$ consists of scheduled \textit{trips}. A trip connects two stations with a specified train via a certain track. A sequence of consecutive trips between two terminal stations is called a {\em line}. A sequence of lines is called a circulation.
Each trip generates two events: a {\em departure event} $v\in\Edep$ from a station,
and an {\em arrival event} $w\in\Earr$ at an adjacent station. These two events are connected by a {\em driving activity} $a=(v,w)$. Together, they model a constraint between those
events, meaning that event $w$ cannot happen before event $v$ has taken place, plus the minimum duration $L_{a}$ assigned to activity $a$.  

\noindent Similar to \cite{schoebel07}, we
distinguish between different types of activities. 

%\todo{
%jede einzeln, turn arcs aufsplitten
%}

\emph{Train activities} $(\Atrain)$ represent the driving, waiting and turning
operations of a train.
\begin{enumerate}
\item {\em Driving activities} ($\Adrive\subset\Edep\times\Earr$) represent scheduled
trips between two stations.
\item {\em Waiting activities} ($\Await\subset\Earr\times\Edep$) represent 
the scheduled waiting
times at a station to let passengers enter and leave the train. 
\item {\em Turning activities}
($\Aturn\subset\Earr\times\Edep$) connect an arrival event with a departure event,
but in addition to waiting arcs, the events connected by a turning activity
do not belong to trips on the same line; instead, they represent the 
possibility for trains to truncate their current circulation and continue on an different line.
%, i.e., to let passengers get 
%off the train, to turn
%(if necessary), and to allow new passengers on board.
\end{enumerate}

\emph{Headway activities} can be split into two subsets of activities that
deal with the limited capacity of tracks between two stations $(\Atrack)$
or on a platform ($\Astation)$:
\begin{enumerate}
\item The headway activities ($\Atrack\subset\Edep\times\Edep$) model the
  headway condition between two events that share the same track resource
  between two stations.
  % they were first introduced in~\cite{schoebel07}. 
  For each $a=(v,w)\in \Atrack$, there is a corresponding activity
  $\tilde{a}=(w,v)\in \Atrack$ so that together they model a precedence
  constraint, i.e., indicate which of the departure events will take place
  first. Thus, only one of each pair of activities can be active.

\item If we have to route trains through a corridor consisting of several
  stations on a single track, we have to introduce a second kind of headway
  activity ($\Astation\subset\Earr\times\Earr$).  These have to make sure
  that trains driving in opposite directions do not enter the same platform
  and block each other; in other words, they ensure that the departure
  event of some train takes place before another train arrives at the same
  platform. Unfortunately, two arrival events could imply more than one
  pair of precedence constraints. In fact, this kind of headway activities
  involves four events, i.e. each possible departure event corresponding to
  the arrival events. In our model we add the possibilities for trains to
  turn at certain points, so the point of time a train may enter a station
  is strongly connected to the direction of the successive trip of the
  train.

  In addition, the headway constraints transmits the priority decision
  taken from the precedence constraints on the
  tracks %induced by the headway activities in $\Atrack$
  of one side of the station to the other side.

  As an example, consider two trains $k,l\in\mathcal{V}$ and three station
  $A,B,C\in\mathcal{S}$. Train $k$ moves from station $A$ through station
  $B$ to station $C$ and train $l$ drives into the opposite directions from
  station $C$ through station $B$ to station $A$. So if the departure event
  from station $A$ to $B$ of train $k$ is scheduled to take place before
  the departure event of train $l$ at station $B$ to $A$, then also train
  $k$ has to depart from station $B$ to $C$ before train $l$ leaves station
  $C$ in opposite directions to station $B$.
  % [$B$-dep, then also [$B$-dep-$k$] has to be scheduled before 
  % [$C$-dep-$l$].
\end{enumerate}
%We need both kinds of the headway activity sets $\Atrack$ and $\Astation$ to model the precedence constraints as we allow trains to perform an early turn at some stations and to make sure, that trains enter and leave stations in a feasible manner without crossing each other. 

In our model trains could be ordered to return to a depot and reserve trains may be inserted.  
Therefore, we extend the event-activity network by adding depots $D_i\in\mathcal{D}$ and replacement capacities $R^{D_i}$ for each depot. For each trip that ends at a station connected to a depot we add a trip to the depot, i.e., a departure event at the station and an arrival event at the depot connected by a driving activity. The sets $\Edepot$ contains the arrival events at the depots and $\Areturn$ the corresponding driving activities. Trains stored in the depots may become reinserted after a minimum idle time, so analogously, the arrival events at a depot become connected to the departure events of possible follow on trips.
% i.e. each arrival event at a depot is connected via a waiting activity and a departure event from the depot to an arrival event at the station and the corresponding departure event of a successive trip.
Similarly, replacement capacities are inserted, by adding an event $r\in\Erdepot$ for each depot supplying replacement vehicles and connecting them to the network.
%via departure events, driving activities and arrival events at the stations to the corresponding departure events of the trips.

%A feasible timetable $\pi\in \mathbb{N}^{\mathcal{E}}$ assigns a time $\pi_v$ to each timetable event $v\in \mathcal{E}.$   

%\todo{PICTURE of an event activity network}\\

Despite rescheduling the timetable events our model should be able to modify the vehicle schedules to find feasible new circulations, 
%by allowing trains to truncate the actual
perform unplanned turns or to return to a depot at certain station. 
Therefore,a classical flow model is integrated. Each activity $a\in\Atrain$ can
transport a flow of at most one unit. For each event $i \in
\mathcal{E}$, the outflow 
equals the inflow. Thus,
every event with an ingoing activity arc that transports one unit of flow has a predecessor
event and due to flow conservation is also connected to a successor event.
% by an outgoing activity arc that also transports one unit of flow. 
The goal is to find feasible flows in the network, i.e., trains always have a sequence of following trips and do not stand idle and block the tracks.

A feasible timetable $\pi\in \mathbb{N}_0^{\mathcal{E}}$ assigns a time $\pi_v$ to each timetable event $v\in \mathcal{E}$ and a feasible flow through the network represents the circulation of trains.

% 
%Thus, we have
%	\[
%		\mathcal{E}=\Edep\cup\Earr\cup\Edepot\cup\Erdepot
%\]
%and
%	\[
%		\mathcal{A}=\Adrive\cup\Await\cup\Aturn\cup\Atrack\cup\Astation.
%\]

%\pagebreak

\subsection{Integer Programming Formulation}

A solution of the following integer programming model yields a new feasible
dispatching timetable $\pi^\ast\in\mathbb{Z}_+^{\mathcal{E}}$ and vehicle
schedule in case of a disruption based on the original timetable $\pi\in
\mathbb{Z}_+^{\mathcal{E}}$ and the corresponding vehicle circulations.

\begin{small}
\begin{equation}
	\max \sum_{v\in \Edep}\sum_{a\in \delta^+(v)}c_a\,y_{a} \label{ip_o1}
\end{equation}
subject to
\begin{align}
	%M(1-f_i)+x_i + \pi_i &\geq \pi_i  																													&&\forall i \in \mathcal{E},\label{ip_c1}\\
	M(1-y_{a})+x_w + \pi_{w} &\geq x_v + \pi_v + L^{min}_{a}  																			&&\forall a=(v,w)\in \Atrain, 																																																											\label{ip_c2}\\
	M(1-y_{a})+x_w + \pi_{w} &\leq x_v + \pi_v + L^{max}_{a}  																		&&\forall a=(v,w)\in \Adrive, \label{ip_c3}\\
	M(3-\sum_{a\in \delta^+(w)}y_{a}-&\sum_{a\in \delta^+(v)}y_{a}-g_{vw}) && \nonumber\\
	+x_w + \pi_w &\geq x_v + \pi_v + L_{vw}			&&\forall (v,w)\in \Atrack,\label{ip_c4}\\
	g_{vw}+g_{wv} &=1		&&\forall (v,w)\in \Atrack,\label{ip_c5}\\
	-M(3-y_{vw}-y_{{v'w'}}-&h_{vv'})  &&  \nonumber\\
	 + x_w + \pi_w + S_{vv'}&\leq  x_{v'} + \pi_{v'} 						&&\forall (v,w),(v',w')\in\Atrain: \nonumber \\
	 & && (v,v')\in \Astation, \label{ip_c6} \\	
	h_{vv'}+h_{{v'v}} &=1																																				&&\forall (v,v')\in \Astation, \label{ip_c7}  \\
\sum_{a\in \delta^-(v)}y_{a} - \sum_{a\in \delta^+(v)}y_{a} &=0																&&\forall v \in \Edep \cup \Earr,\label{ip_c8}\\
\sum_{a\in \delta^-(v)}y_{a} - \sum_{a\in \delta^+(v)}y_{a}&\leq 1														&&\forall v \in \Edepot,\label{ip_c9}\\
\sum_{a\in \delta^+(v)}y_{a} &\leq R^{D_i}																													&&\forall v \in \Erdepot,  D_i\in\mathcal{D},\label{ip_c10}\\
%f_v &= \sum_{a\in \delta^+(v)}y_{a}																							&&\forall v \in \mathcal{E}_{drive} \cup \mathcal{E}_{wait} \cup  \mathcal{E}_{depot},\label{ip_c11}\\
	%f_v &\in \{0,1\}		
	y_a &\in \{0,1\}																																							&&\forall a\in \mathcal{A},\label{ip_c12}\\
x_v &\in \mathbb{Z}_+																																					&&\forall v\in \mathcal{E},\label{ip_c13}\\
	g_{a} 	 &\in \{0,1\}																																					&&\forall a\in \Atrack,\label{ip_c14}\\
	h_{a} &\in \{0,1\}																																				&&\forall  a\in \Astation\label{ip_c15}
\end{align}
\end{small}

The variables of the model are as follows: $y$ is the circulation of
trains, $x_v$ determines the delay of event~$v$, $g_a$ and $h_a$ represent
precedence constraints (see below). In the following, we successively
discuss the constraints of the above model.

The objective function \eqref{ip_o1} aims at maximizing the number of trips
that still will be served in the re-optimized dispatching timetable
$\pi^\ast$, possibly with delay. Because the frequency of most subway
systems is quite high, canceling few trips will lead to only a minor delay
and inconvenience for the passengers. Therefore, minimizing the overall
delay, as it is done in most publications about delay management, is only a
secondary goal.

A trip is served if the corresponding departure event $v\in\Edep$ has an
outflow of one unit of flow, i.e., $a=(v,w)\in\Adrive$ is part of a train
circulation. We can attach an additional cost coefficient $c_a$ to each
driving activity $a$; this allows a weighting of the trips.

Given an original timetable $\pi$ and a corresponding vector
$x\in\mathbb{Z}_+^{\mathcal{E}}$, representing the delays $x_v$ of each
event $v$, Constraints~(\ref{ip_c2}) require that if an activity $a=(v,w)$
is in the solution, the earliest time for event $w$ to start is the
starting time of the predecessor event $\pi_v$, plus the occurred delay
$x_v$ and the minimum duration $L^{min}_{a}$ of activity $a$. With respect
to the buffer times, a delay $x_v$ of event $v$ might cause a delay $x_w$
of the successor event $w$. We assume that the driving times are symmetric.
Constraints~\eqref{ip_c3} make sure that a driving activity $a$ is bounded
by a maximal duration $L^{max}_{a}$, which includes possible buffer times
and safety margins, but should not be too large to prevent trains being
idle for a long period of time.
The precedence constraints formulated by (\ref{ip_c4}) and (\ref{ip_c5})
for shared track resources, come in pairs for $(v,w),(w,v)\in\Atrack$ and
ensure that conflicting events $v,w\in\Edep$ that use the same track
resource are correctly scheduled to prevent deadlocks and keep safety
margins. The model implies
\[ 
g_{vw}= \left\{ \begin{array}{r@{\quad : \quad}l} 
				 1 & \text{if event } v \text{ takes place
                                   before } w, \\ 0 & \text{otherwise,}
                               \end{array} \right. 
\]
where $L_{vw}$ represents the minimal waiting time for event $w$ to start
if $w$ is scheduled after $v$. If the trips corresponding to the departure
events $v,w\in\Edep$ use the track in the same direction, $L_{vw}$ is the
safety margin between both departure events; if the track is used in
opposite direction, $L_{vw}$ contains the maximum allowed transit time for
the trip and a safety margin $S_{vw}$.
%$L^{max}_{\hat{a}}$ for $\hat{a}=(v,v')\in\Adrive$ and the safety margin $S_{\stackrel{\leftrightarrow}{a}}$.\\

Similarly the precedence constraints (\ref{ip_c6}), (\ref{ip_c7}) for the
tracks inside a station ensure that no two trains use the same platform at
the same time, where the model implies
 \[ 
h_{vv'}= \left\{ \begin{array}{r@{\quad : \quad}l}
				 1 & \text{if event } w \text{ takes place
                                   before event, } v' \\ 0 &
                                 \text{otherwise.} \end{array} \right. 
\]
Regarding two conflicting arrival events $v,v'\in\Earr$ at some station and
their corresponding departure events $w,w'\in\Edep$ with $a = (v,w)$,
$\tilde{a} = (v',w') \in \Atrain$ denoting the activities, if the solution contains both activities ($y_a =
y_{\tilde{a}} = 1$) and $h_{vv'}=1$, departure event $w$ takes place before
the arrival event $v'$ at the platform with respect to the safety margin
$S_{vv'}$ (see Figure~\ref{h-station}).

The safety margins $S_{vw}$ and $S_{vv'}$ should be equal and depend on the
direction of the corresponding trips:

% \\ $S_{vw} > 1/2 \bl max
% [(L^{max}_{wz}-L^{min}_{wz}),(L^{max}_{uv}-L^{min}_{uv})]$ with
% $(w,z),(u,v')\in \mathcal{A}$ denoting the corresponding induced driving or
% turn activity.

\begin{figure}
  \centering
  \includegraphics[width=.6\linewidth]{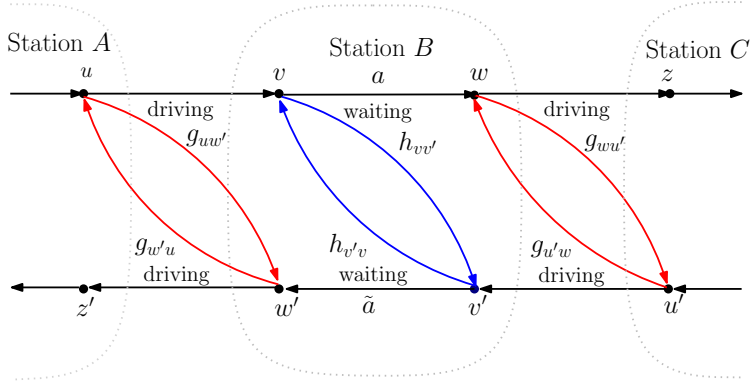}
  \caption{{The event-activity network for a conflict on a single
      track. One train is driving from station A to station C while another
      train is driving in the opposite direction. The nodes represent the
      corresponding departure and arrival events and $a,\tilde{a}$ the
      waiting activities. Furthermore the disjunctive headway activities
      are shown, denoted by the variable name of corresponding constraint
      of the IP.}}
  \label{h-station}
\end{figure}

\begin{lemma}
  The safety margin $S_{vv'}$ for two trains driving in opposite
  directions that enter and leave a station on the same track should
  satisfy
  \[
  S_{vv'} > \tfrac{1}{2}\, \max
  \{(L^{max}_{wz}-L^{min}_{wz}),(L^{max}_{uv}-L^{min}_{uv})\}.
  \]
\end{lemma}

\begin{proof}
  The event activity network for two trains, driving through a station in
  opposite directions on a single track is shown in Figure~\ref{h-station}.
  In this case, we have $y_a = y_{\tilde{a}} = 1$, i.e., the solution
  contains these conflicting trips. To simplify notation, we set $\pi_v^* =
  \pi_v + x_v$ for all $v\in\mathcal{E}$, and we use the new times of a
  dispatching timetable. In order to obtain a feasible timetable, the
  headway variables $g_{uw'}$, $h_{vv'}$, $g_{wu'}$ should be equal. We
  show that
  \[
  S_{vv'} > \tfrac{1}{2}\, \max
  \{(L^{max}_{wz}-L^{min}_{wz}),(L^{max}_{uv}-L^{min}_{uv})\}
  \]
  implies $g_{uw'}=h_{vv'}=g_{wu'}=1$.
  \renewcommand{\labelenumi}{(\roman{enumi})}
  \begin{enumerate}
  \item $h_{vv'}=1 \bl \Rightarrow \bl g_{wu'}=1$\\
    If $h_{vv'}=1$ then $\pi_w^* + S_{vv'} < \pi_{v'}^*$, because
    of~(\ref{ip_c6}). We have $L_{u'w} = S_{u'w} + L^{max}_{u'v'} = S_{vv'}
    + L^{max}_{u'v'}$. If $g_{wu'} = 0$, it follows that
    \begin{align*}
      \pi_w^*  &\geq \pi_{u'}^*+L_{u'v'}^{min}+S_{vv'}  \\
      \pi_w^* + S_{vv'} &\geq \pi_{u'}^*+L_{u'v'}^{min}+2S_{vv'} \\
      \pi_{v'}^* &\geq \pi_{u'}^*+L_{u'v'}^{min}+2S_{vv'}\\
      \pi_{v'}^* + L_{u'v'}^{max} &\geq \pi_{u'}^*+ L_{u'v'}^{max}+L_{u'v'}^{min}+2S_{vv'}\geq \pi_{v'}^*+L_{u'v'}^{min}+2S_{vv'} \\
      L_{u'v'}^{max} &\geq L_{u'v'}^{min}+2S_{vv'}\\
      S_{vv'}& \leq \tfrac{1}{2} \, (L_{u'v'}^{max} - L_{u'v'}^{min}),
    \end{align*}
    which is in contradiction to the condition on $S_{vv'}$.
	
  \item $g_{wu'}=1 \bl \Rightarrow \bl h_{vv'}=1$\\
    Assume $g_{wu'}=1$ and $h_{vv'}=0$. Because of~(\ref{ip_c4}), we have
    $\pi_{u'}^* \geq \pi_w^* + L_{wz}^{min} + S_{vv'}$. Because
    of~(\ref{ip_c6}) and $h_{vv'} = 0$, it follows that $\pi_{u'}^* \geq
    \pi_{v'}^*+ S_{vv'} + L_{wz}^{min} + S_{vv'}$, which is in
    contradiction to the condition of $S_{vv'} > 0$.
  \item $h_{vv'}=1 \bl \Rightarrow \bl g_{uw'}=1$ and
  \item $ g_{uw'}=1 \bl \Rightarrow \bl h_{vv'}=1$ can be shown
    analogously.\hfill$\Box$
  \end{enumerate}
\end{proof}

\begin{figure}
  \centering
  \includegraphics[width=.6\linewidth]{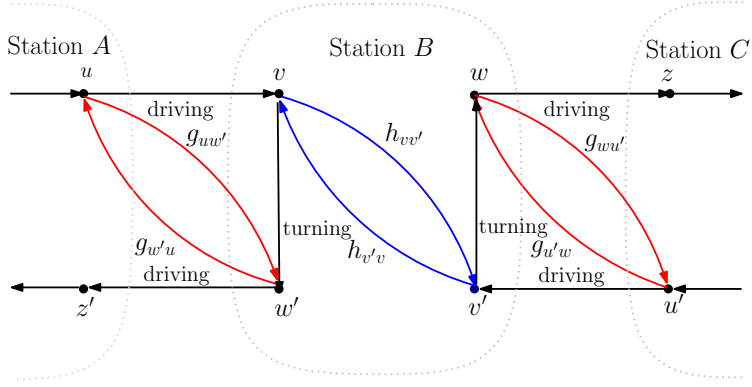}
  \caption{{Example: Trains are allowed to turn at~$B$}}
  \label{turn}
\end{figure}

Remember that we cannot always express the headway constraints via a single
variable, as the condition $g_{uw'}=h_{vv'}=g_{wu'}$ does not hold, if
trains are allowed to turn (see Figure~\ref{turn}).

As we described above, a flow model is integrated to model the circulation
of trains. Constraint (\ref{ip_c8}) ensures flow conservation on the
events, i.e., the outflow should be equal to the inflow. A flow in the
network encodes the sequence of trips for each train during the observed
time interval. Note that trains represent a single commodity. Depots can
consume flow (\ref{ip_c9}) and initiate flow~(\ref{ip_c10}) with respect to
the maximum number of replacement trains~$R^{D_i}$ of depot~$D_i\in\mathcal{D}$.

The big-$M$ has to be chosen sufficiently large in order to yield a correct
model; analogous to~\cite{schachtebeck09}, we set
\[ 
M := Y+ \max_{(v,w)\in \mathcal{A} } (\pi_w - \pi_v + L_{vw}+S_{vw}), 
\]
where $Y$ denotes the maximum allowed delay for each event and we let
$L_{vw} = S_{vv'} = 0$, on arcs where they have not been defined. In our
case $Y$ is bounded by the cycle time. 

We summarize:
\begin{observ}
  Feasible dispatching timetables $\pi^\ast = \pi + x$ correspond precisely
  to the feasible solutions of the integer program with constraints
  (2)--(14).
\end{observ}

\subsection{Reducing the Size of the Integer Program}

The computational running time for solving the IP is mainly influenced by
the number of precedence constraints in the system. These constraints
involve big-$M$ constants and thus lead to a weak LP relaxation: here we
can set each $g_a$ and $h_{a}$ with $a\in\Ahead$ to $\tfrac{1}{2}$. This
makes it desirable to reduce the number of precedence constraints by fixing
them wherever possible. This can be achieved by considering the scheduling
logic inherent in the system. If a disruption occurs, decisions for parts
of the network imply decisions for other parts, e.g., if there are
alternating sequences of driving and waiting activities with no turning
possibilities. Furthermore, trains driving into the same direction cannot
pass each other. Thus, variables can be fixed by using the natural
precedence relations. In a second step, we can build blocks of such
sequences and contract the driving and waiting activities into just one
driving activity with the new minimum driving time. Because the delay of an
event is passed to successive events, the travel time is bounded by the
maximum allowed delay with respect to the buffer times. Safety intervals
are inherited from the original event-activity network.  The cost of the
departure events is set to the sum of the corresponding contracted
departure events. This process yields a \emph{reduced event-activity
  network} $\mathcal{N}^r = (\mathcal{E}^r, \mathcal{A}^r)$.

\begin{lemma}
  An instance of the reduced network
  $\mathcal{N}^r=(\mathcal{E}^r,\mathcal{A}^r)$ has a feasible solution, if
  and only if the corresponding instance of the original event-activity
  network $\mathcal{N}=(\mathcal{E},\mathcal{A})$ has a feasible solution,
  and their objectives are equal.
\end{lemma}

\def\Eseq{\mathcal{E}_{\text{seq}}^{vw}}

\begin{proof}
  Let $\pi^r$ be a feasible dispatching timetable for the reduced
  event-activity network $\mathcal{N}^r = (\mathcal{E}^r,\mathcal{A}^r)$.
  Let $a = (v,w)\in\mathcal{A}^r$ be a driving activity, and let
  $\Eseq\subset\Edep\cup\Earr$ be the corresponding sequence of events from
  the start event $v\in\Edep$ to the end event $w\in\Earr$. If $x_v \geq
  x_w$ in~$\mathcal{N}^r$, we reduce the delay of the events along the
  sequence as late as possible; this is necessary to prevent blocking of
  following trains. If $x_v \leq x_w$, we increase the delay as soon as
  possible to prevent a collision with the train in front. The other
  direction follows analogously.\hfill$\Box$
\end{proof}

Note that the overall delay of a dispatching in the the original
network~$\mathcal{N}$ that results from the application of the strategy in
the above proof might be higher than the minimal possible delay
in~$\mathcal{N}$. This imposes no serious restriction, since our objective
is insensitive to delays.

The reduced network $\mathcal{N}^r$ usually contains significantly fewer
nodes and arcs, resulting in a smaller IP instance. This enables us to
obtain solutions for significantly larger networks; in our experiments, the
number of binary variables was reduced by roughly a factor of three.

\section{Experiments}

In this section we report on computational experiments, which we conducted
to evaluate the applicability of our models.  They are based on the
real-life timetable of the subway line U6 in Vienna (Figure~\ref{u6}). We employed the Falko
tool of Siemens AG~\cite{falko}. Timetables are given with a granularity of
one second.

The considered subway line has 24 stations, with terminal stations
Siebenhirten and Floridsdorf; the regular travel time between these
terminals is 34 minutes. We use typical frequencies of 5 or 10 minutes in
our experiments. Trains are circulating between the terminals. The stations
Floridsdorf, Michelbeuern, and Alterlaa are connected to depots that each
contain a reserve train in our scenarios. Sidings exist at Siebenhirten.
Our initial timetable contains buffer times of about 10\% of the scheduled
driving times between stations and at the end of the line. 

We consider a variety of different disruptions with focus on the blockage
of one side of the two tracks. Such disruptions often yield very hard
instances and lead to severe problems in the daily operation of a line.
During a disruption, a section of tracks on one side is blocked for a time
interval between 5 minutes and 2 hours; until the section is reopened, no
train is allowed to enter the blocked section. The track topology allows
switching the track and performing a turn at several stations. Trips
affected by the blocked section can use the unblocked track in the opposite
direction. More specifically, we allow trains to pass the switches that are
immediately before and after the blocked section. The circulation of a
train can be truncated by introducing a turning.
%Obviously, trains sharing the same partially blocked track in opposite
%directions incur a number of dependencies between trips.

For each scenario below, we use our model to generate a feasible
dispatching timetable. We have to return to the original timetable within 60
minutes after reopening the blocked sections. Choices for trains include
turning at specified stations, increasing their delays, or returning to a
depot; the maximum allowed delay of a trip is equal to the cycle time.

\begin{figure}
  \centering
  \includegraphics[width=.6\linewidth]{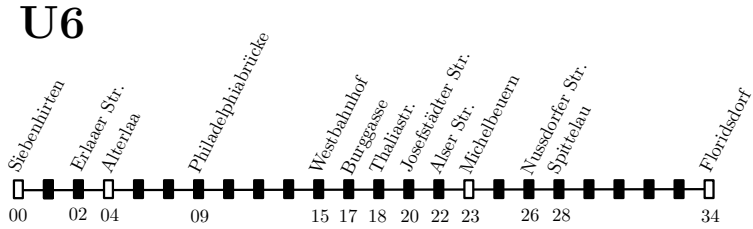}
  \caption{{The Vienna subway line U6. Boxes represent the stations: a white box, if the station is connected to a depot or siding. The numbers below are the driving times from Siebenhirten to the stations.}}
  \label{u6}
\end{figure}

\subsection{Results}
We tested four scenarios on the Vienna U6 that differ in location, transit
times, the topology of the switches and turn possibilities to evaluate 
the impact of disruptions:
\begin{description}
\item[Scenario 1:] A disruption occurs between Michelbeuern and
  Westbahnhof. These stations are located close to the center of the line
  and are heavily used. Trips between these stations are assigned to the
  opposite track. Trains running from Michelbeuern to Westbahnhof have to
  switch the track after Michelbeuern and return to their originally
  assigned track between Burggasse and Westbahnhof. We allow trains to turn
  at Michelbeuern and Westbahnhof. The driving time between these stations
  is about 7 minutes and the trains have to share the same track in 4
  stations.
\item[Scenario 2:] A disruption occurs between Philadelphiabr\"ucke and
  Erlaaer Stra{\ss}e that are located near the terminal station
  Siebenhirten. Trains have to switch the track in front of the station
  Philadelphiabr\"ucke and back again before arriving at Erlaaer
  Stra{\ss}e. We allow trains to turn at these stations.  The driving time
  is 7 minutes and the trains share the same track in 4 stations. The depot
  Alterlaa is located between those stations.
\item[Scenario 3:] A disruption occurs between Nu{\ss}dorfer Stra{\ss}e and
  Michelbeuern. Trains have to switch shortly after Nu{\ss}dorfer
  Stra{\ss}e and return to their original track before arriving at
  Michelbeuern. Options for turning are given at Michelbeuern and (for
  trains going to Floridsdorf) after Spittelau. The driving time through
  the bottleneck is 3 minutes but points for turning are not close to the
  bottleneck section.
\item[Scenario 4:] Here, a disruption occurs between Alser Stra\ss e and
  Thalia\-stra\ss e, so that trains have to switch the track after
  Michelbeuern and return after Josephst\"adter Stra{\ss}e. The driving
  time between these stations is 5 minutes and the options for turning are
  given at Michelbeuern and Westbahnhof, which are not close to the
  disrupted track section.
\end{description}
For our experiments, we used the IP solver CPLEX 12.1 on a PC with a 3.0
GHz Intel Core2Duo CPU and 2 Gbytes of RAM.

The tables 1 -- 6 are organized by scenarios and cycle frequencies; rows
correspond to disruption times in minutes, indicated in the first column.
Columns~2 shows the number of binaries that correspond to the precedence
constraints and the flow on the activity arcs. All computations were
performed using the reduced network. Column~3 gives the number of general
variables for the delays. The fourth column shows the original number of
trips in the undisrupted timetable during the observed time window.  Column~5 shows the number of trips in the new dispatching timetable, derived with
CPLEX in a maximum runtime of 1800 seconds, while the sixth column provides
the best bound provided by CPLEX.  Column~7 shows the necessary solution
time in seconds, or (in case of delays of 45 minutes or more) 1800 if CPLEX
could not establish optimality of a solution withing 1800 seconds. Column 8
indicates whether optimality could be proved, or the gap with
respect to an upper bound. The final column~9 shows the number of trips in
the disposition timetable, if we stop CPLEX finding a solution in only 60
seconds and succeed in finding at least a feasible solution. For this
computation, the MIP node selection strategy of CPLEX is set to best
estimate search.

\begin{table}[p]
%  \fontsize{7}{10}\selectfont
  \footnotesize
  \centering
  \caption{\textsc{Scenario 1: disruption between Michelbeuern and
      Westbahnhof, cycle time of 5 minutes}}
  \begin{tabular*}{\linewidth}{@{\extracolsep{\fill}}lr r r r r r r r@{}}\toprule
    dur & bin & int & $|\mathcal{A}_{dep}|$ & sol & u.bound & time & status  & 60s\\ \midrule
    5   & 628 & 306 &       171             & 165 &  165.00 & 0.25 & opt     & 165\\ 
    10  & 745 &	331 &	    177		    & 171 &  171.00 & 3.41 & opt     & 171\\ 
    15  & 861 & 356 &	    198	            & 183 &  183.00 & 0.87 & opt     & 183 \\ 
    20  & 978 &	381 &       207	            & 189 &  189.00 & 9.42 & opt     & 189 \\ 
    30	& 1215& 431 &  	    231		    & 207 &  207.00 &24.88 & opt &  	207\\
    45	& 1578& 506 &	267		    & 231 &  231.00 &141.56&	opt  &  228\\  
    60	& 950 & 581 &	303		    & 252 &  269.63 &1800  & 7.0\%   & 252 \\ 
    90	&2721 & 731 &	375	            & 276 &  351.14 &1800  &27.2\%   & 	--\\ 
    120	&3528 & 881 &   447	            & --  &  424.62 &1800  & --      & --\\ \bottomrule
  \end{tabular*}
  \label{tab_1}
\end{table}

\begin{table}[p]
  % \fontsize{7}{10}\selectfont
  \footnotesize
  \centering
  \caption{\textsc{Scenario 1: disruption between Michelbeuern and Westbahnhof, cycle time of 10 Minutes}}
  \begin{tabular*}{\linewidth}{@{\extracolsep{\fill}}lr r r r r r r r@{}}\toprule
    dur & bin & int & $|\mathcal{A}_{dep}|$ & sol & u.bound &time & status  & 60s\\ \midrule
    5         &	254	     &150	      &  85	                  &85	       &    85.00	 & 0.08	& opt &  		85				\\ 
    10        &	307	     &164	      &92	                    &92	       &92.00	     &0.09	 &opt &  		92				\\ 
    15				&365	     &175	      &97	                    &97	       &97.00	     &0.89	  &opt &  		97				\\  
    20  				&416       &189				&104	                  &104	     &104.00	     &0.13	  &opt &  	104				\\ 
    30			&	524					&	214			&	116										&	110				&110.00	&4.04			&opt &  	110				\\ 
    45			&694				&250				&	133									&	127					&127.00	&231.72		&opt &  	127				\\ 
    60			&	854				&289	      &152	                 &146	        &146.00	 &559.87	 &opt &  140				\\ 
    90			&	1193			&364				&188									&170					&188.00	&1800	     &10.5\% &  170  		\\ 
    120			&1453	      &439	      &224	                &200	        &224.00	 &1800	    &12.0\%    & 197			\\ \bottomrule
  \end{tabular*}
  \label{tab_2}
\end{table}

\begin{table}[p]
  %\fontsize{7}{10}\selectfont
  \footnotesize
  \centering
  \caption{\textsc{Scenario 2: disruption between Philadelphia\-br\"ucke and Erlaaer Str., cycle time of 5 Minutes}}
  \begin{tabular*}{\linewidth}{@{\extracolsep{\fill}}lr r r r r r r r@{}}\toprule
    dur & bin & int & $|\mathcal{A}_{dep}|$ & sol & u.bound &time & status  & 60s\\ \midrule
    5					&	819			&			327				&	172							&		162				&	169.00	&1.51	& opt &  162\\ 
    10				&	913			&353						&184							&184					&184.00				&0.28	 & opt & 184\\ 
    15				&1092			&	379						&	196							&183					&183.00		&5.30			&	opt  &	183	\\  
    20				&1231			&405						&208							&192					&192.00		&7.64			& 	opt  &  192 \\ 
    30				&1515			&	457						&	232							&	209					&209.00		&	59.11		& opt & 209 \\ 
    45				&1956			&	535					 &	268							&	236					&243.00		&	1800		&2.9\%			 & 235  \\ 
    60				&2415			&613						&685							&258					&284.31		&1800			&10.2\%      & 252\\ 
    90				&	2721	  &731	         &375	              &276	        &351.14	   &1800	  &27.2\% & 272\\ 
    120  	    &3528	    &881	          &447		          &   --        &424.62	  &1800    	&--   & -- \\ \bottomrule
    % 180				&5250			&1181						&	591							&   - 				&568.58		&1800    	&	-    &  -  \\ \hline		
  \end{tabular*}
  \label{tab_3}
\end{table}

\begin{table}[p]
  %\fontsize{7}{10}\selectfont
  \footnotesize
  \centering
  \caption{\textsc{Scenario 2: disruption between Philadelphia\-br\"ucke and Erlaaer Str., cycle time of 10 Minutes}}
  \begin{tabular*}{\linewidth}{@{\extracolsep{\fill}}lr r r r r r r r@{}}\toprule
    dur & bin & int & $|\mathcal{A}_{dep}|$ & sol & u.bound &time & status  & 60s\\ \midrule
    5					&	348			&	158				&	86									&86						&86.00			&0.08	&		opt   &  86  \\
    10			&	411				&177				&92										&92						&92.00		&	0.09	& opt		&  92\\ 	
    15				&473			&184				&98										&98						&	98.00	 &0.14	& opt		&  98\\ 
    20				&511			&203				&104									&104					&104.00		&0.90		&opt &  104\\ 
    30				&623	    &229	      &116	                 &116        &	116.00	  &0.12	   &opt & 116\\
    45				&863		&262				&134										&130				&	130.00		&24.65	& opt &  130\\
    60				&1065			&307			&152										&146				&146.00			&591.15	& opt &  145\\
    90				&1482			&385			&188										&180				&184.87			&1800		& 2.7\%		&  174\\ 
    120					&1917			&463			&224										&	209				&221.00			&1800		&5.7\%		 & 209 \\ \bottomrule
    % 180					&2841			&619			&	296										&270				&296.00			&1800	  &9.6\%			& 255 \\ \hline
  \end{tabular*}
  \label{tab_4}
\end{table}

\begin{table}[p]
  %\fontsize{7}{10}\selectfont
  \footnotesize
  \centering
  \caption{\textsc{Scenario 3: disruption between Nu{\ss}dorfer Stra{\ss}e and Michelbeuern, cycle time of 5 Minutes}}
  \begin{tabular*}{\linewidth}{@{\extracolsep{\fill}}lr r r r r r r r@{}}\toprule
    dur & bin & int & $|\mathcal{A}_{dep}|$ & sol & u.bound &time & status  & 60s\\ \midrule
    5	&786	&431	&224	&224	&224.00&	0.19	&opt	&224 \\ 
    10 &	891 &	462 &	238 &	233 &	233.00	 &0.77	& opt	& 233 \\ 
    15 &	999	&393	&252	&247	&247.00	&2.12	&opt& 	247 \\ 
    20 &	1109 &	524     &266    &261	& 261.00&	3.00 &	opt &	261 \\ 
    30 &	1335 &	568	& 294	&282	&282.00	&158.54 &opt	&283 \\ 
    45 &	1698&	679	&336	&317	&330.00	&1800	&4.1\%	&330 \\ 
    60 &	2061&	772	&378	&352	&372.00	&1800	&5.6\%	&372 \\ 
    90 &	2859&	958	&462	&411	&462.00	&1800	&12.4\%&	462 \\
    120 &	3729&	1114	&546	&467	&546.00	&1800	&16.9\%&	546 \\ \bottomrule
    % 180 &	5685&	1516	&714	&589	&714.00	&1800	&21.2\%&	- \\ \hline
  \end{tabular*}
  \label{tab_5}
\end{table}

\begin{table}[p]
  %\fontsize{7}{10}\selectfont
  \footnotesize
  \centering
  \caption{\textsc{Scenario 4: disruption between Alser Stra{\ss}e and Thaliastra{\ss}e, cycle time of 5 Minutes}}
  \begin{tabular*}{\linewidth}{@{\extracolsep{\fill}}lr r r r r r r r@{}}\toprule
    dur & bin & int & $|\mathcal{A}_{dep}|$ & sol & u.bound &time & status  & 60s\\ \midrule
    5 &				798&	416	&224	&224&	224.00&	0.17	& opt	& 224\\ 
    10& 935&	450	&240	&232	&232.00	&3.16	&opt&	232 \\ 
    15&1072	&484	&256	&238	&238.00	&97.90	&opt&	238 \\ 
    20& 1211&	518	&272	&254	&254.00	&71.65&	opt&	272 \\ 
    30&1495&	586	&304	&276	&293.12	& 1800&	6.7\%	&276 \\ 
    45&1936&	688	&352	&316	&343.96	& 1800&	8.8\%	&316 \\ 
    60&2395&	790	&425	&351	&398.00	& 1800&	13.6\%	&338 \\ 
    90&3367&	994	&496	&419	&495.22	& 1800&	18.1\%	& -- \\ 
    120&4441&	1198&	592	&481	&592.00	& 1800&	23.0\%	& -- \\ \bottomrule
    % 180&6715&	1606	&784&	571	&784.00	& 1800&	37.3\%  & - \\ \hline
  \end{tabular*}
  \label{tab_6}
\end{table}

\subsection{Observations}

We were able to find feasible solutions for most of the instances,
including the relatively difficult scenarios with long disruption times. In
most cases, these solutions were achieved quickly, while the bulk of the
work was invested in establishing optimality. This indicates that our
method should be suitable in even larger real-life situations in which a
fast solution is needed and proving its optimality is a secondary concern.
As described above, the LP relaxation is weak. Indeed, it has a large gap
compared to the IP formulation. As a consequence the upper bounds are
improved late during the solution process. If these bounds become more
important, it may be of interest to establish additional inequalities.

\begin{figure}
  \centering
  \includegraphics[width=.6\linewidth]{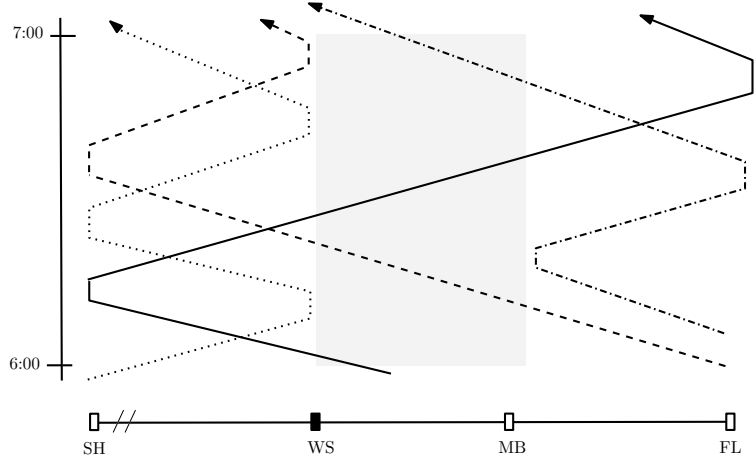}
  \caption{{An excerpt of the time and space diagram of a new dispatching timetable. The diagram shows 4 trains circulating between the
      terminal stations Siebenhirten (SH) and Floridsdorf (FL). Due to a disruption between 6.00 a.m.\ and 7.00
      a.m.\ only one track between Westbahnhof (WS) and
      Michelbeuern (MB) can be used. Thus, the schedules of the trains represented by the dashed and dotted lines become modified.  
%  A snapshot of the time and space diagram generated by
%      FALKO of the new disposition timetable and vehicle
%      schedule %for a 60 minutes disruption
%      in scenario 1. The diagram shows the trains circulating between the
%      terminal stations Siebenhirten (HT) and Floridsdorf (FL) from 5.00
%      a.m.\ untill 9.50. a.m. Due to a disruption between 6.00 a.m.\ and 7.00
%      a.m.\ only one track can be used between Westbahnhof (WS) and
%      Michelbeuern (MB). The highlighted line represents a specific train
%      that first drives through the bottleneck section while later performs
%      an early turn at (WS).
}
}
  \label{time_space}
\end{figure}

For practical purposes, the actual structure of the resulting vehicle
schedules is of interest. An excerpt of a time and space diagram for the re-optimized
timetable and vehicle schedule for a disruption of 60 minutes in Scenario~1
is shown in Figure~\ref{time_space}. 

The utilization of the tracks is generally high. But in the computed
solutions, trains are still passing through the bottleneck section. The
frequency, however, is heavily decreased.  The new vehicle schedule
contains 20 early turn operations. All of the three replacement trains are
used, in order to realize the new timetable.

\subsection{Extensions}

Every change in the vehicle schedule needs to be distributed quickly to the
involved persons. Furthermore, truncated circulations, caused by early
turns or vehicles returning early to a depot, lead to necessary transfers
for the passengers. Thus, it is desired to keep the number of transfers in
the vehicle schedule reasonably low. To achieve this, we add a second term to the objective function by
penalizing early turns and early returns to a depot:
\begin{equation}\label{ip_o2}
  \max \sum_{v\in \Edep}\sum_{a\in \delta^+(v)}c_a\, y_{a} -
  \sum_{b\in\Aturn\cup\Areturn}c_b\, y_{b}.
\end{equation}
Each unplanned turning or early return to a depot penalizes the objective
value by~$c_b$. This results in a trade-off between the number of recovered
weighted trips and the transfers for passengers.

In first experiments, we were able to find solutions with significantly
less transfers for the passengers, but only a few more canceled trips. Fine
tuning the weight coefficients $c_a$ and $c_b$ is important, since they
have big impact on the design of the circulations provided by solution.
%
%Further experiments show that reducing the number of reserve trains may lead to no integer solution could be found in time. In these cases, it seems to be helpful to stretch the time horizon until the original timetable should be recovered. A time horizon that is set very short after the expected end of the disruption may also lead to bad solutions or an infeasible problem.   

We are optimistic that the interaction with practitioners may help in
designing better models, with preferences for particular solution types.
For example, if the minimal transfer time of the disrupted section is quite
small compared to the cycle frequency, delaying the affected trips could be
sufficient. Furthermore, computed solutions often have an alternating
structure: the disposition timetable often builds clusters of trips (with
respect to their safety distance) and let these clusters alternate through
the bottleneck. On the other hand, if the transfer time is too large, just
delaying trips is not sufficient, trips have to be dropped and the
solutions often contain shuttle trains. Finally, allowing trains to turn
within the bottleneck yields some good solutions.

\section{Conclusions}

We have introduced techniques for integrated rescheduling of trips and
vehicles for real-time disruption management of rail-based public
transportation systems, especially for subway systems.  Using an IP formulation
and appropriate reduction techniques, we were able to achieve very good
solutions for a variety of test scenarios arising from a real-world subway
line. Tests show that our feasible solutions are always optimal or close to
being optimal, indicating the practical usefulness of our method.

The spectrum of further improvements includes fine-tuning of our IP/LP
approach and exploiting the structure of the underlying networks to reduce the
solution space in advance. We would like to deal with larger-scale
networks, for examples (subway) systems that do not have separated track
system for each line.

% Given that we are in close interactions with practitioners, we are
% optimistic that further progress in modeling and computation will be
% achieved.

\section*{Acknowledgment}
\noindent
Martin Lorek gratefully acknowledges support from Siemens AG by a research
fellowship. We thank the members of the Rail Automation Department, Mobility
Division of Siemens AG, in particular York Schmidtke, Karl-Heinz Erhard and Gerhard Ruhl,
for many helpful conversations and an ongoing fruitful collaboration,
and for the use of Falko.
%\end{acknowledgment}

\bibliographystyle{asmems4}
\bibliography{paper}

\end{document}